\documentclass[12pt]{article}
\usepackage{amssymb,xcolor,epic}
\usepackage{amsmath}
\usepackage{graphicx}
\usepackage{enumerate}
\usepackage{mathrsfs}
\usepackage{latexsym}
\usepackage{psfrag}
\usepackage{mathrsfs}
\usepackage{multicol}
\usepackage{algorithm}
\usepackage{algpseudocode}

\newcommand{\qed}{\hfill $\square$}

\newenvironment{proof}{\noindent{\em Proof}.}{\qed\bigskip}

\newtheorem{theorem}{Theorem}[section]
\newtheorem{lemma}[theorem]{Lemma}
\newtheorem{definition}[theorem]{Definition}
\newtheorem{proposition}[theorem]{Proposition}

\newtheorem{conjecture}{Conjecture}[section]
\newtheorem{question}{Question}[section]

\newcommand{\comments}[1]{}
\begin{document}

\title{Shifted-antimagic Labelings for Graphs}

\author{
Fei-Huang Chang
\thanks{Division of Preparatory Programs for Overseas Chinese Students
National Taiwan Normal University New Taipei City, Taiwan
{\tt Email:cfh@ntnu.edu.tw} supported by MOST 106-2115-M-003-005}
\and
Hong-Bin Chen
\thanks{Department of Applied Mathematics
National Chung-Hsing University
Taichung City, Taiwan
{\tt Email: andanchen@gmail.com} supported by MOST 105-2115-M-035-006-MY2}
\and
Wei-Tian Li
\thanks{Department of Applied Mathematics
National Chung-Hsing University
Taichung City, Taiwan
{\tt Email:weitianli@nchu.edu.tw} supported by MOS T105-2115-M-005-003-MY2}
\and
Zhishi Pan
\thanks{Department of Mathematics
Tamkang University
New Taipei City, Taiwan
{\tt Email:zhishi.pan@gmail.com} }
}

\date{\small \today}

\maketitle

\begin{abstract}
The concept of antimagic labelings of a graph is to produce distinct vertex sums by labeling edges through consecutive numbers starting from one. A long-standing conjecture is that every connected graph, except a single edge, is antimagic. Some graphs are known to be antimagic, but little has been known about sparse graphs, not even trees.

This paper studies a weak version called $k$-shifted-antimagic labelings which allow the consecutive numbers starting from $k+1$, instead of starting from 1, where $k$ can be any integer. This paper establishes connections among various concepts proposed in the literature of antimagic labelings and extends previous results in three aspects:
\begin{itemize}
\item Some classes of graphs, including trees and graphs whose vertices are of odd degrees, which have not been verified to be antimagic are shown to be $k$-shifted-antimagic for sufficiently large $k$.
\item Some graphs are proved $k$-shifted-antimagic for all $k$, while some are proved not for some particular $k$.
\item Disconnected graphs are also considered.
\end{itemize}
\end{abstract}

\section{Introduction}

Graph labeling problems are interesting and broadly studied.
%This paper studies labelings, called {\em antimagic labelings}, on edges of simple and finite graphs.
The concept of {\em antimagic labelings} was first introduced by Hartsfield and Ringel~\cite{HR1990}.
Here is the definition.

\begin{definition}\label{DEF:AML}
Let $G=(V,E)$ be a graph with $|E(G)|=m$,
and let $f$ be an injective function $f$ from $E(G)$ to some subset $S\subset\mathbb{R}$.
The vertex sum of a vertex $v$, induced by $f$, is defined to be $\phi_f(v):=\sum_{e\sim v} f(e)$, where $e\sim v$ means the edge $e$ is incident to $v$.
If $S=\{1,2,\ldots, m\}$ and the vertex sums are all distinct for all vertices,
then we say $G$ is antimagic and $f$ is an antimagic labeling of $G$.
\end{definition}

It is trivial that this kind of labelings does not exist if $G=K_2$.
Hartsfield and Ringel \cite{HR1990} proved that some graphs are antimagic, including the paths $P_n$, the cycles $C_n$,
and the complete graphs $K_n$ for $n\ge 3$, and came up with the following two conjectures.

\begin{conjecture}{\rm \cite{HR1990}}\label{CONJ:ATG}
Every connected graph with at least three vertices is antimagic.
\end{conjecture}

\begin{conjecture}{\rm \cite{HR1990}}\label{CONJ:ATT}
Every  tree other than $K_2$ is antimagic.
\end{conjecture}

The former conjecture has been proved to be true for graphs with many edges or graphs whose vertex degrees are equal.
In~\cite{AKLRY2004}, Alon, Kaplan, Lev, Roditty, and Yuster proved that dense graphs are antimagic. Precisely, they proved that a graph $G$ with minimum degree $\delta(G)\ge c\log |V|$ for some constant $c$, or with maximum degree $\Delta(G)\geq |V(G)|-2$ is antimagic.
Moreover, complete partite graphs except for $K_2$ are antimagic. The antimagic labeling of regular graphs
has been investigated by many groups of researchers~\cite{W2005,W2008,C2009,LLT2011,LZ2014,CLZ2015},
and is completely solved recently by B\'{e}rczi, Bern\'{a}th, and Vizer~\cite{BBV2015},
and by Chang, Liang, Pan, and Zhu~\cite{CLPZ2016}, independently.

Generally speaking, in contrast to dense graphs, little has been known about the antimagic labeling of sparse graphs.
The most celebrated result on trees is given by Kaplan, Lev, and Roditty~\cite{KLR2009}
(with a small error in the proof corrected by Liang, Wong, and Zhu~\cite{LWZ2012}) that every tree with at most one vertex of degree two is antimagic.
Other classes of trees that satisfy Conjecture~\ref{CONJ:ATT} but not mentioned previously
include the spider graphs~\cite{S2015,H2015}, the double spider graphs \cite{CCP2017}, and the caterpillars with some extra assumptions~\cite{LMS2017}.

Wang and Hsiao~\cite{W2008} introduced an analogue concept, called $k$-antimagic labelings, in the context of antimagic labelings on the Cartesian product and the lexicographic product of sparse graphs. In a $k$-antimagic labeling, they label the edges by $k+1$ through $k+|E(G)|$ for a nonnegative integer $k$ rather than by $1$ through $|E(G)|$, and still ask for distinct vertex sums. By translating the labels of edges of an antimagic graph, they were able to produce antimagic labelings of the Cartesian product of some graphs. Notice that only nonnegative integers $k$ are considered in their paper.
%Regardless the original purpose in~\cite{W2008}, the study of $k$-antimagic labeling of graphs is interesting in its own.
%In the same paper, Wang and Hsiao also proved some properties related to the $k$-antimagic labeling.
We also make a remark about a different $k$-antimagic labeling, by Hefetz~\cite{H2005} and Wong and Zhu~\cite{WZ2012}, which allows $k$ more labels $|E(G)|+1,\ldots, |E(G)|+k$, i.e., the labels are starting from 1 through $|E(G)|+k$.

This paper follows the definition of $k$-antimagic labelings in~\cite{W2008} and generalizes it to a broader setting where $k$ can be any integer.
To avoid confusion between the $k$-antimagic labelings in~\cite{H2005,WZ2012} and in \cite{W2008}, the labelings considered in the paper will be called {\em $k$-shifted-antimagic labelings}.
\begin{definition}
Let $G$ be a graph with $|E(G)|=m$. Given $k\in\mathbb{Z}$,
if there exists an injective function $f$ from $E(G)$ to $\{k+1,k+2,\ldots,k+m\}$
 such that the vertex sums $\phi_f(v)$ are all distinct for all vertices $v\in V(G)$,
then we say $G$ is $k$-shifted-antimagic and $f$ is a $k$-shifted-antimagic labeling of $G$.
\end{definition}

Regardless the original purpose in~\cite{W2008}, the study of $k$-shifted-antimagic labelings of graphs is interesting in its own. Obviously, it is a natural generalization of the traditional antimagic labelings. Indeed, an antimagic graph is $0$-shifted-antimagic. Now that the two conjectures by Hartsfield and Ringel are still far from resolved, one may naturally wonder if they can be proved $k$-shifted-antimagic for some $k$. This paper gives an affirmative answer to the $k$-shifted-antimagic version for some classes of graphs that have yet been proven antimagic, including trees and graphs whose vertex degrees are odd.

An interesting feature of $k$-shifted-antimagic labelings is that simply shifting the used labels by one in a $k$-shifted-antimagic labeling graph does not guarantee a $(k+1)$-shifted-antimagic labeling.  Interestingly, if a graph $G$ is $k$-shifted-antimagic for some $k$, then it can be shown that there exists an integer $k(G)$ such that $G$ is $k'$-shifted-antimagic for any $|k'|\geq k(G)$.

A naturally raised question is to characterize the spectrum of integers $k$ of a given graph such that it is $k$-shifted-antimagic. On the one hand, some examples are fully characterized in this paper. However, the provided examples are only small graphs due to the difficulty of determining a graph not $k$-shifted-antimagic for a certain $k$. On the other hand, there exist some graphs that are $k$-shifted-antimagic for all $k$. We call such a graph {\em absolutely antimagic}. Wang and Hsiao noticed that a graph $G$ is $k$-shifted-antimagic for all $k\geq 0$ if there is an antimagic labeling $f$ of $G$ such that $\phi_f(u)>\phi_f(v)$ whenever $\deg(u)>\deg(v)$. Such a labeling $f$ and graph $G$ were called {\em strongly antimagic} later in \cite{H2015}. Obviously, regular graphs are strongly antimagic and in fact they are absolutely antimagic by a straightforward argument. This paper provides a non-regular example by proving that the paths $P_n$, $n\geq 6$, are absolutely antimagic.

Disconnected graphs are also taken into account in this paper, although they have been received little attention in the literature. One reason could be the existence of disconnected graphs that are not antimagic, and proving not antimagic for a graph is relatively difficult because it needs to go through every possible labeling. Shang et al. \cite{SLL2015} first pointed out the fact that a graph is barely antimagic if it contains too many components isomorphic to $P_3$. Precisely, they proved that the union of a star $S_n$ and $c$ copies of $P_3$ is 0-shifted-antimagic if and only if $c\le \min\{2n+1,\frac{2n-5+\sqrt{8n^2-24n+17}}{2}\}$. Shang \cite{S20152} proved that the graph consisting of $c$ copies of $P_3$ is antimagic if and only if $c=1$. We extend their results by showing that for any graph $G$ there exists an integer $c$ such that the union of $G$ and $c$ copies of $P_3$ is not antimagic, and providing a necessary and sufficient condition for the graph of $c$ copies of $P_3$ is $k$-shifted-antimagic. In addition, we demonstrate that the graphs $2P_4$ and $2S_3$ are not $k$-shifted-antimagic if and only if $k=-2$ and $k=-5$. These are the only examples found so far that the integers $k$ for which the graph is not $k$-shifted-antimagic do not appear consecutively.

%\begin{question}\label{Q1}
%Given a graph $G$, for which integer $k$ is $G$ $k$-shifted-antimagic?
%\end{question}

%Of particular note is that $k$ can be nonpositive and disconnected graphs are considered.
%This paper answers Question~\ref{Q1} for regular graphs with degree greater than one, forests, and graphs without vertex of even degrees.
%We prove that a forest or a graph without vertices of even degrees is $k$-shifted-antimagic for $|k|$ sufficiently large. This implies that trees are $k$-shifted-antimagic for some $k$.
%We show that regular graphs and paths of length greater than five are $k$-shifted-antimagic for any integer $k$, and find some examples that are not $k$-shifted-antimagic for some particular $k$.  Moreover, we manage to determine all values of $k$ for which the graph is not $k$-shifted-antimagic for some graphs.

The rest of the paper is organized as follows.
In Section 2, we present results on the regular graphs, the forests, and the graphs without vertices of even degrees.
Section 3 demonstrates some graphs that are not $k$-shifted-antimagic, and characterizes the spectrum of values of $k$ for which a given graph is not $k$-shifted-antimagic.
These graphs include the trees of diameter at most four and some disconnected graphs.
We present some remarks and open problems in the last section.

\section{The $k$-shifted-antimagic labeling}

%From now on, we always assume our graph $G$ contains $n$ vertices and $m$ edges, if not specified.

%Let us begin with a small example $P_3=uvw$. There are two injective mappings from $E(P_3)$ to $\{1,2\}$, which are
%all antimagic. In particular, the weight sum of $v$ must be 3. If we add $k$ to each label, then the weight

%Let $G$ be a graph.
%Assume that there exists a labeling $f$ on $E(G)$
%such that for any pair of distinct vertices $u$ and $v$ with
%$\deg(v)=\deg(u)$, we have $\phi_f(v)\neq\phi_f(u)$.
%Although $f$ does not guarantee that
%$\phi_f(v)\neq \phi_f(u)$ for $u$ and $v$ of different degrees, if we translate the function $f$ to $f'=f+M$ by some constant $M$ large enough, then we will have not only
%$\phi_{f'}(v)<\phi_{f'}(u)$ for any vertices $u$ and $v$ with $\deg(u)<\deg(v)$, and also
%$\phi_{f'}(v)\neq\phi_{f'}(u)$ for any vertices $u$ and $v$ with $\deg(u)=\deg(v)$.
%This is because a vertex of larger degree has the larger translation of its vertex sum, and two vertices of the same degree have the same translations.
%The above simple observation leads to the following lemma.

This section starts with a lemma due to a simple observation that after increasing all labels by 1 in a labeling of a graph, the more degree a vertex has the more increments it gets.

\begin{lemma}\label{LEM:SFT}
Given a graph $G$, if there exists an injective function $f$ from $E(G)$ to $\{1,2,\ldots,m\}$ such that
$\phi_f(v)\neq \phi_f(u) $ whenever $\deg(v)=\deg(u)$ for distinct vertices $u$ and $v$, then $G$ is $k$-shifted-antimagic for any sufficiently large $k$.
\end{lemma}

\begin{proof}
Let $f$ be a function satisfying the required condition. Consider $f'=f+k$,
where $k\ge (m-1)(\Delta(G)-1)$ and $\Delta(G)$ is the maximum degree of $G$.
For two vertices $u$ and $v$ with $\deg(u)<\deg(v)$, we have
\begin{eqnarray*}
&&\phi_{f'}(v)-\phi_{f'}(u)\\
&=& [k\deg(v)+\phi_f(v)]-[k\deg(u)+\phi_f(u)]\\
&\ge&
[(m-1)(\Delta(G)-1)\deg(v)+\phi_f(v)]-[(m-1)(\Delta(G)-1)\deg(u)+\phi_f(u)]\\
&\ge&[(m-1)(\Delta(G)-1)(\deg(v)-\deg(u))-(m-1)\deg(u)]+(\deg(v)-\deg(u))\\
&>&0.
\end{eqnarray*}
Clearly, $\phi_{f'}(v)\neq \phi_{f'}(u) $ if $\deg(v)=\deg(u)$ for distinct vertices $u$ and $v$.
Hence, $f'$ is a $k$-shifted-antimagic labeling.
\end{proof}

We call a labeling with the property described in Lemma~\ref{LEM:SFT} an {\em SDDS-labeling} (same-degree distinct-sum).
Although it is unclear about the existence of the
SDDS-labeling for general graphs, we have a method to construct it for various graphs.
The following ``level-by-level labeling algorithm'' is originally used by Cranston, Liang, and Zhu~\cite{CLZ2015} to find the antimagic labelings of odd regular graphs.
Their ideas are sketched briefly in the following.
%{\bf (The level-by-level labeling algorithm)}
First, pick a vertex $w$ from the given graph $G$.
Partition $V(G)$ into levels $L_0,L_1,\ldots,L_d$, where $L_i=\{u\mid d(w,u)=i\}$
and $d$ is the furthest distance of a vertex from $w$.
Let $G[L_i]$ and $G[L_i,L_{i-1}]$ be the subgraph induced by $L_i$ and the bipartite subgraph induced by the two parts $L_i$ and $L_{i-1}$, respectively.
Then construct a labeling $f$ by labeling the edges in $G[L_d]$, $G[L_d,L_{d-1}]$, $\ldots$ , $G[L_1]$ and $G[L_1,L_0]$ in order using the smallest unused labels with some additional rules.
We use this method to get the following two theorems.

%Then let $E_i$, $E_i'$, and $E_i''$ be
%$E(G[L_i])$, $E(G[L_i,L_i-1])-\sigma(L_i)$, and $\sigma(L_i)$, respectively.
%We construct a labeling $f$ by labeling the edges in $E_d$, $E_d'$, $E_d''$, $E_{d-1}$, $E_{d-1}'$, $E_{d-1}''$, $\ldots$, $E_1$, $E_1'$, $E_1''$ in order using the smallest unused labels with some rules.

%\begin{description}
%\item{1.} Label the edges in $E_i$ arbitrary by the allowed labels..
%\item{2.} For $E'_i$, let $G'_i=G[L_i,L_{i-1}]-E(T)$ and $O_i=\{v\mid v\in V(G'_i),\deg(v)\mbox{ is odd}\}$.
%Then we partition $E'_i$ into $|O_i|/2$ open trials and some closed. ({\bf To be written})
%\item{3.} To label the edges in $E_i''$, recall that every vertex $v\in L_i$ is incident to a unique edge $\sigma(v)\in E_i''$, and observe that all edges, except for $\sigma(v)$, incident to $v$ have been labeled.
%Define the partial vertex sum
%\[\phi'_f(v):=\sum_{e\sim v, e\neq \sigma(v)}f(e).\]
%We may assume $\phi'_f(v_1)\le \phi'_f(v_2)\le\cdots$ for vertices $v_1,v_2,\ldots \in L_i$.
%Then we give the usable labels to $\sigma(v_1), \sigma(v_2), \ldots$, in ascending order.
%\end{description}

%To label edges in $E_i'$, we first partition it into disjoint union of $E()$
%We need to label the edges in $E_i'$ carefully. ({\bf To be written})

\begin{theorem}
If $G$ is a forest without a component isomorphic to $K_2$, then $G$ is $k$-shifted-antimagic for sufficiently large $k$.
\end{theorem}

\begin{proof} We prove this theorem by showing that every forest admits an SDDS-labeling. It suffices to prove the statement for $G$ being a tree
since if we label the components by nonoverlapping intervals of integers,
then two vertices of the same degree in different components must have distinct vertex sums.

Let $G$ be a tree. Then $E(G[L_i])$ is empty for each $i$, and
we only need to label the edges in $G[L_i,L_{i-1}]$.
Let $|E(G[L_i,L_{i-1}])|=e_i$ for $i=d,d-1,\ldots,1$.
Now we define an $SDDS$-labeling $f$ on $E(G)$.
First, label the edges in $G[L_d,L_{d-1}]$ arbitrarily by $1,\ldots, e_d$.
Suppose that all the edges in $G[L_d,L_{d-1}]\cup\cdots\cup G[L_i,L_{i-1}]$ for some $i\le d$ are labeled.
For each vertex $v\in L_{i-1}$, there exists a unique edge $e_v\in G[L_{i-1},L_{i-2}]$ incident to $v$.
Define the partial vertex sum of $v$ by
\[\phi'_f(v):=\sum_{e\sim v, e\neq e_v}f(e).\]
Without loss of generality, we may assume $\phi'_f(v_1)\le \phi'_f(v_2)\le \cdots \le \phi'_f(v_{e_{i-1}})$
for vertices $v_1,v_2,\ldots, v_{e_{i-1}}\in L_{i-1}$.
Finally, label $e_{v_1},e_{v_2},\ldots$ by $1+\sum_{l=i}^{d}e_l, 2+\sum_{l=i}^{d}e_l, \ldots, \sum_{l=i-1}^{d}e_l$,  accordingly.

Observe that the above labeling method promises that $\phi_f(u)\neq \phi_f(v)$ for $u,v\in L_{i}$ and $\deg(v)=\deg(u)$.
Next, we show $\phi_f(u)\neq \phi_f(v)$ for $u\in L_{i}$, $v\in L_{j}$, $i< j$, and $\deg(v)=\deg(u)$.
If $\deg(u)=\deg(v)=1$, then $\phi_{f}(u)\neq \phi_f(v)$ as $v$ and $u$ are not adjacent for otherwise $G=K_2$.
Suppose $\deg(u)=\deg(v)\ge 2$, $u\in L_{i}$, $v\in L_{j}$, and $i< j$.
If $i= 0$, then $u$ is the root $w$, which is adjacent to the largest $\deg(u)$ labels.
Thus, $\phi_f(u)> \phi_f(v)$ follows.
Assume that $0<i<j$. Observe that one of the edges incident to $u$ is in $G[L_i,L_{i-1}]$
and others are in $G[L_{i+1},L_i]$.
Similarly, one of the edges incident to $v$ is in $G[L_j,L_{j-1}]$
and others are in $G[L_{j+1},L_j]$.
Since all the labels of edges in $G[L_i,L_{i-1}]$ ($G[L_{i+1},L_i]$) are greater than
all the labels of edges in $G[L_j,L_{j-1}]$ ($G[L_{j+1},L_j]$),
one can conclude that $\phi_f(u)>\phi_f(v)$.
Hence, $f$ is an $SDDS$-labeling.
\end{proof}

\begin{theorem}\label{THM:ODD}
If a graph $G$ consists of vertices of odd degrees and contains no component isomorphic to $K_2$, then $G$ is $k$-shifted-antimagic for sufficiently large $k$.
\end{theorem}

\begin{proof}
%As with the explanation in the paragraph of the previous theorem, we only need to deal with connected grpahs.
Similarly, it suffices to prove the statement for connected graphs.
We first pick a vertex $w$ in $G$ and define $L_i$'s, $G[L_i]$'s, and $G[L_i,L_{i-1}]$'s accordingly as before.
For each bipartite graph  $G[L_i,L_{i-1}]$,
we shall look for an injection $\sigma_i$ from $L_i$ to $E(G[L_i,L_{i-1}])$ via $\sigma_i(v)=vu\in E(G)$,
with $v\in L_i$ and $u\in L_{i-1}$, which satisfies the following properties:
\begin{description}
\item{1.} The graph $G[L_i,L_{i-1}]-\sigma_i(L_i)$ can be decomposed into edge-disjoint open trails $T_i$'s, where an
open trail $T_i$ is a sequence of vertices $v_1v_2\ldots v_l$ for some $l$, $v_iv_{i+1}\in E(G)$, $v_iv_{i+1}\neq v_jv_{j+1}$ for $i\neq j$,
and the initial vertex $v_1$ is not the same as the terminal vertex $v_l$.
\item{2.} No two trails share the initial or terminal vertices.
In other words, if $T_1=v_1v_2\ldots v_a$ and $T_2=u_1u_2\ldots u_b$, then $\{v_1,v_a\}\cap\{u_1,u_b\}=\emptyset$.
\end{description}
Indeed, the existence of such an injection has been proved by Cranston, Liang, and Zhu~\cite{CLZ2015} in their ``{\em Helpful Lemma}''.
%({\bf To be written})

Suppose that the injections $\sigma_i$'s for all $L_i$'s have done. Let $\sigma_i(L_i)=\{\sigma_i(u)\mid u\in L_i \}$.
We construct a labeling $f$ by assigning the smallest unused labels to the edges in $G[L_d]$, $G[L_d,L_{d-1}]-\sigma_d(L_d)$,
$\sigma_d(L_d)$, $G[L_{d-1}]$, $G[L_{d-1},L_{d-2}]-\sigma_{d-1}(L_{d-1})$, $\sigma_{d-1}(L_{d-1})$, $\ldots$ in order.
\begin{itemize}
\item For edges in $G[L_i]$, we arbitrarily assign the usable labels to the edges.

\item For edges in $G[L_i,L_{i-1}]-\sigma_d(L_i)$, by the properties of $\sigma_i$, we decompose it into edge-disjoint open trails.
These trails can be classified as three types: $W$-type, $M$-type, and $N$-type.
A trail is of $W$-type, or $M$-type, or $N$-type if its initial and terminal vertices are both in $L_{i-1}$,
or both in $L_i$, or one in $L_{i-1}$ and the other in $L_i$, respectively.
The strategy of labeling each trail is to make the partial vertex sum of an internal vertex in $L_{i-1}$
not smaller than that of an internal vertex in $L_i$.
Assume that the usable labels of edges in $G[L_i,L_{i-1}]-\sigma_d(L_i)$ are $s,s+1,\ldots, \ell$.
We first label the trails of $W$-type and $M$-type. Notice that each trial of the two types have an even number of edges.
Suppose that $2r$ labels are already used. To label the edges of an unlabeled trail $v_1v_2\cdots v_a$, if it is of $W$-type,
then assign $s+r, \ell-r, s+r+1, \ell-r-1, \ldots$ to $v_1v_2,v_2v_3,v_3v_4,v_4v_5,\ldots$ successively;
else if it is of $M$-type, then  assign $\ell-r, s+r, \ell-r-1, s+r+1, \ldots$ to $v_1v_2,v_2v_3,v_3v_4,v_4v_5,\ldots$  successively.
To deal with trails of $N$-type, we first arrange them in pairs with possibly one trail left.
Pick $T_1=v_1v_2\ldots v_a$ and $T_2=u_1u_2\ldots u_b$ a pair of trails of $N$-type.
We may assume that $v_1\in L_i$ and $u_1\in L_{i-1}$.
Then we assign labels $\ell-r, s+r, \ell-r-1, s+r+1, \ldots, \ell-r-\frac{a}{2}+1$ to $v_1v_2,v_2v_3,v_3v_4,v_4v_5,\ldots,v_{a-1}v_a$
and labels $s+r+\frac{a}{2}-1, \ell-r-\frac{a}{2}, s+r+\frac{a}{2}, \ell-r-\frac{a}{2}-1,\ldots,2r+a+b-2$ to $u_1u_2,u_2u_3,u_3u_4,u_4u_5,\ldots,u_{b-1}u_b$ successively.
If there is a trail of $N$-type left at the end, then we label it with the same strategy as labeling $T_1$.

\item For edges in $\sigma_i(L_i)$, observe that for each vertex $u\in L_i$, $\sigma_i(u)$ is the only unlabeled edge incident to $u$.
As before, we define the partial vertex sum of $u$ by
\[\phi'_f(u):=\sum_{e\sim v, e\neq \sigma_i(u)}f(e),\] and assign the usable labels to edges in $\sigma_i(L_i)$ with
$f(\sigma_i(u_1))<f(\sigma_i(u_2))$ if and only if $\phi'_f(u_1)\le \phi'_f(u_2)$ for $u_1,u_2\in L_i$.
\end{itemize}

Finally, we verify that $f$ is an $SDDS$-labeling.
By the labeling rule for $\sigma_i(L_i)$'s at the end of the previous paragraph,
two vertices of the same degree in the same level must have distinct vertex sums.
%It is clear that two vertices of the same degree in the same level must have distinct vertex sums
%from the rules we label $\sigma_i(L_i)$'s.
Suppose that $u\in L_i$ and $v\in L_j$ with $i>j$ and $\deg(u)=\deg(v)=2d+1$.\\
If $i\ge j+2$, then $\phi_f(u)<\phi_f(v)$ is straightforward since the label of any edge incident to $u$
is less than that of any edge incident to $v$.\\
Consider the case of $i=j+1$.
Again, suppose that the labels of edges in $G[L_i,L_{i-1}]-\sigma_i(L_i)$ are $s,s+1,\ldots,\ell$.
It suffices to show that $\phi'_f(u)\le d(s+\ell)\le \phi'_f(v)$ since $\sigma_i(u)<\sigma_{i-1}(v)$.
Observe that all edges incident to $u$ are in $G[L_{i+1},L_i]$, $G[L_i]$, or $G[L_i,L_{i-1}]$,
and the edges in $G[L_{i+1},L_i]$ or $G[L_i]$ have labels less than $s$.\\
If $\deg_{G[L_i,L_{i-1}]}(u)$ is odd, then except for the edge $\sigma_i(u)$,
other edges can be paired  so that each pair appears successively in a trail and
contributes either $s+\ell $ or $s+\ell-1$ to the vertex sum of $u$.
Hence $\phi'_f(u)$ is at most $d(s+\ell)$.\\
If $\deg_{G[L_i,L_{i-1}]}(u)$ is even, then except for the edge $\sigma_i(u)$,
the remaining $\deg_{G[L_i,L_{i-1}]}(u)-1$ edges can be paired so that each pair appears successively in a trail
and one remaining edge is the initial or terminal edge of a trail.
For each paired edges, they contribute either $s+\ell $ or $s+\ell-1$ to the vertex sum of $u$
as before, and for the single edge, it contributes at most $\ell$ to the vertex sum of $u$.
Since $\deg(u)$ is odd, there exists one edge in $G[L_{i+1},L_i]\cup G[L_i]$ incident to $u$, which has label less than $s$.
By the discussion above, it is easily seen that $\phi'_f(u)\le d(s+\ell)$.

The idea of showing $d(s+\ell)\le \phi'_f(v)$ in the following is similar.
All edges incident to $v$ are in $G[L_i,L_{i-1}]$, $G[L_{i-1}]$, or $G[L_{i-1},L_{i-2}]$,
and edges in $G[L_{i-1}]$ or $G[L_{i-1},L_{i-2}]$ have labels greater than $\ell$.
If $\deg_{G[L_i,L_{i-1}]-\sigma_i(L_i)}(v)$ is even,
then these edges can be paired  so that each pair of edges appears successively in a trail and
contributes either $s+\ell $ or $s+\ell+1$ to the vertex sum of $v$.
Moreover, each edge in $\sigma_i(L_i)$ incident to $v$ has label greater than $\ell$.
So $\phi'_f(u)$ is at least $d(s+\ell)$.
Next, if $\deg_{G[L_i,L_{i-1}]-\sigma_i(L_i)}(v)$ is odd,
then these edges can be paired so that each pair of edges appears successively in a trail
and one remaining edge is the initial or terminal edge of a trail.
For each paired edges, they contribute either $s+\ell $ or $s+\ell+1$ to the vertex sum of $v$
as before, and for the single edge, it contributes at least $s$ to the vertex sum of $v$.
Therefore, the partial vertex sum of $v$ is at least $d(s+\ell)$.
This completes the proof.
\end{proof}

%\begin{theorem}
%Every graph $G$ is $k$-shifted-antimagic for all sufficiently large $k$.
%\end{theorem}

%\begin{proof}
%\end{proof}

%\begin{corollary}
%Every graph $G$ with $m$ edges and maximum degree $\Delta(G)\ge 2$ is $m\Delta(G)$-shifted-antimagic.
%\end{corollary}

%\begin{proof}
%\end{proof}

%A mapping $f$ from $E(G)$ to a subset of integers is a {\em strong labeling} if the weight sums satisfying
%$\phi_f(u)<\phi_f(v)$ whenever $\deg(u)< \deg(v)$.

Next, we turn our attention to graphs which are strongly antimagic, i.e., an antimagic labeling $f$ satisfying $\phi_f(u)<\phi_f(v)$ whenever $\deg(u)<\deg(v)$.
If a graph $G$ is strongly antimagic, then the {\em join} of $G$ and a vertex $v$, $G\vee v$,
obtained by connecting every vertex to $v$ is also strongly antimagic by simply giving the largest $|V(G)|$ labels to edges incident to $v$.
Moreover, spider graphs~\cite{H2015,S2015} and double spiders graphs are strongly antimagic~\cite{CCP2017}.
A simple observation by Wang and Hsiao \cite{W2008} is that the strongly-antimagic property implies the $k$-shifted-antimagic property for all $k\geq 0$.

\begin{proposition}{\rm \cite{W2008}}\label{PROP:W2008}
If $G$ is a strongly antimagic graph, then it is $k$-shifted-antimagic for all $k\ge 0$.
\end{proposition}
\begin{proof}
The proof follows from the same argument in the proof of Lemma~\ref{LEM:SFT}.
\end{proof}

Obviously, it changes the sign of each vertex sum by changing the sign of the label of each edge. By symmetry, we have the following simple fact.

\begin{proposition}\label{PROP:SYM}
If $G$ is a $k$-shifted-antimagic graph , then it is $-(m+k+1)$-shifted-antimagic.
\end{proposition}

\begin{proof}
Assume that $f$ is a $k$-shifted-antimagic labeling of $G$. Define $g$ from $E(G)$ to $\{-(k+m+1)+1,\ldots, -(k+m+1)+m\}$ by letting $g(e)=-f(e)$.
Then $\phi_g(v)=-\phi_f(v)$ are distinct for all vertices $v\in V(G)$.
Hence, $G$ is $-(m+k+1)$-shifted-antimagic.
\end{proof}

Are there graphs $k$-shifted-antimagic for any integer $k$? So far we have seen several examples of graphs that are $k$-shifted-antimagic  when $k$ is sufficiently large or sufficiently small.
In addition, if $G$ is strongly antimagic, then by Proposition\ref{PROP:W2008} and Proposition~\ref{PROP:SYM}
there are at most $m$ possible values of $k$ for which $G$ is not $k$-shifted-antimagic. This property shall have an advantage when proving a graph absolutely antimagic, i.e., $k$-shifted-antimagic for any integer $k$. In the following, two classes of graphs are found to be absolutely antimagic. A trivial example is regular graphs, which are shown to be antimagic in \cite{BBV2015,CLPZ2016} and thus absolutely antimagic as no two vertices are of different degrees in a regular graph. We demonstrate absolutely antimagic graphs that are not regular.

\begin{theorem}\label{THM:PAA}
Every path $P_n$ with $n\ge 6$ is absolutely antimagic.
\end{theorem}

\begin{proof}
{\bf Claim:} Every path $P_n$ with $n\ge 3$ is strongly antimagic.

In fact, a path can be viewed as a special type of spider,
which is shown strongly antimagic~\cite{S2015}.
For the completeness, we present the proof of the claim.

\noindent{\em Proof of the Claim}.
We give a strongly antimagic labeling directly. Denote a path on $n$ vertices by $P_n=v_1v_2\cdots v_n$. For odd $n$,
let
\[f(v_iv_{i+1})=\left\{
\begin{array}{rl}
1,&\mbox{ for }i=1;\\
i+1, &\mbox{ for }2\le i\le n-2;\\
2,&\mbox{ for }i=n-1.
\end{array}
\right.\]
Note that except $\phi_f(v_2)=4$, $\phi_f(v_{n-1})=n+1$, and $\phi_f(v_{n})=2$, other vertex sums are
distinct odd integers. For even $n$, let
\[f(v_iv_{i+1})=\left\{
\begin{array}{rl}
1,&\mbox{ for }i=1;\\
n+1-i, &\mbox{ for }2\le i\le n-1.\\
\end{array}
\right.\]
Then the vertex sums are
distinct odd integers except for $\phi_f(v_2)=n$ and $\phi_f(v_{n})=2$.
\qed

Consider $P_n=v_1v_2\cdots v_n$ with $n\ge 6$ and $e_i=v_iv_{i+1}$.
By the claim and Proposition~\ref{PROP:W2008}, $P_n$ is $k$-shifted-antimagic for any $k\ge 0$.
Also, it is $k$-shifted-antimagic for $k\le -n$ by Proposition~\ref{PROP:SYM}.
To prove that $P_n$ is $k$-shifted-antimagic for $-n+1\le k\le -1$, it is sufficient to verify those $k$'s with $-n/2\le k\le -1$ by the symmetry property in Proposition~\ref{PROP:SYM}.
For each $k$ in this range, we define a labeling $f_k$ as follows.
For $-n/2\le k\le -3$ or $k=-1$, let $f_k(v_{|k|}v_{|k|+1})=0$.
Now consider the two subpaths $Q_1=v_1v_2\ldots v_{|k|}$ and  $Q_2=v_{|k|+1}\ldots, v_{n-1}v_n$.
By the same labeling method in the claim, we use $1,2,\ldots, n-|k|-1$ to label $E(Q_2)$ while using $1,2,\ldots,k-1$ to label $E(Q_1)$ but  adding a negative sign to each label in $E(Q_1)$.
Then the vertex sums at vertices in $V(Q_i)$ are distinct by the claim. Moreover, the vertex sums are all negative at $V(Q_1)$ and positive at $V(Q_2)$.

The last case is $k=-2$. Then for odd $n$,
let
\[f_{-2}(v_iv_{i+1})=\left\{
\begin{array}{rl}
-1,&\mbox{ for }i=1;\\
1, &\mbox{ for }i=2;\\
0, &\mbox{ for } i=3;\\
i-2,&\mbox{ for } i\ge 4.
\end{array}
\right.
\]
We have vertex sums $-1,0,1,2,5,7,9,\ldots,2n-3,n-1$ at $v_1,\ldots,v_n$, respectively.
For even $n$, define
\[f_{-2}(v_iv_{i+1})=\left\{
\begin{array}{rl}
0,&\mbox{ for }i=1;\\
-1, &\mbox{ for }i=2;\\
n-i,&\mbox{ for }i\ge 3.
\end{array}
\right.\]
The vertex sums are $0,-1,n-4,2n-7,2n-9,\ldots,3,1$ at $v_1,\ldots,v_n$, respectively.
The proof is completed.
\end{proof}

A path $P_n$ is not $k$-shifted-antimagic for any integer $k$ when $n=2$,
while it is $k$-shifted-antimagic for any integer $k$ if $n\ge 6$.
For $n\in\{3,4,5\}$, $P_n$ is not $k$-shifted-antimagic for some integers $k$.
In Section 3, we will present several classes of graphs, including $P_3$, $P_4$, and $P_5$,
which are not $k$-shifted-antimagic for some $k$.

\section{Graphs that are not $k$-shifted-antimagic}

This section focuses on graphs that are not $k$-shifted-antimagic.
We first show some connected graphs, specifically trees, that are not $k$-shifted-antimagic with certain values of $k$. These cases of trees shall be classified according to diameter.
Then for some examples of disconnected graphs we characterize the sufficient and necessary conditions of the value $k$ for which they are $k$-shifted-antimagic.
Despite the fact that every connected graph $G$ can be antimagic, we prove that there exists a disconnected graph $G'$ containing $G$ as a component but $G'$ is not antimagic.

%Conjecture~\ref{CONJ:ATG} is equivalent to the statement that every connected graph is $k$-shifted-antimagic for $k=0$. Although this conjecture is still widely open, we can prove that some connected graphs are not $k$-shifted-antimagic with some specific values of $k$. Moreover, we can show that every connected graph $G$ is contained in a graph $G'$ such that every component of $G'$ is antimagic but $G'$ is not antimagic.
%We present the results in the following two subsections.

\subsection{Trees with diameter at most four}

The diameter of a tree $G$ is the largest integer $d$ so that $G$ contains a path $P_{d+1}$ as a subgraph.
The path $P_2$ is the only tree of diameter one, which is not $k$-shifted antimagic for any $k$.
A tree of diameter two is also called a star $S_n$, where $n\ge 2$ is the number of leaves.
\begin{proposition}
For $n\ge 2$, the star $S_n$ is $k$-shifted-antimagic if and only if
$k\not\in\{-\frac{1}{2}n-1, -\frac{1}{2}n\}$ for even $n$ and $k\neq -\frac{1}{2}(n+1)$ for odd $n$.
\end{proposition}

\begin{proof}
 Let $f$ be an injection from $E(S_n)$ to $\{k+1,\ldots,k+n\}$.
Then we have vertex sums $\phi_f(v)=k+1,\ldots, k+n$ for the leaves,
 and $\phi_f(v)=\sum_{i=1}^n (k+i) = nk+\frac{1}{2}(n^2+n)$ for the unique internal vertex.
Thus, $f$ is not $k$-shifted-antimagic if and only if
\[k+1\le nk+\frac{1}{2}(n^2+n)\le k+n.\]
Solving the inequalities yields $-\frac{1}{2}n-1\le k\le -\frac{1}{2}n$.
\end{proof}

A tree of diameter three must contain a path $P_4$ with all remaining vertices adjacent to one of the two internal vertices of $P_4$.
We call it a {\em double star} $S_{a,b}$,
where $a, b\ge 1$ are the numbers of leaves adjacent to the two interval vertices, respectively.
Without loss of generality, assume $a\ge b$ in $S_{a,b}$.

\begin{theorem}
A double star $S_{a,b}$ is absolutely antimagic if and only if
$(a,b)\not\in \{(2,1)\}\cup\{(a,1)\mid a\mbox{ is odd}\}$. Moreover,
 $S_{2,1}$ is $k$-shifted-antimagic if and only if $k\not\in\{-2,-3\}$,
and $S_{a,1}$ is $k$-shifted-antimagic if and only if $k\neq -\frac{1}{2}(a+3)$.
\end{theorem}

\begin{proof} Denote the vertices of $S_{a,b}$ by $v$, $u$, $v_1,\ldots, v_a$, and $u_1,\ldots,u_b$ such that
$v$ and $u$ are the internal vertices, and $v_i$'s and $u_j$'s are the leaves adjacent to $v$ and $u$, respectively.
Fix an integer $k$ and let $n_+$ and $n_-$ be the number of positive and negative integers in $\{k+1,k+2,\ldots,k+a+b+1\}$, respectively.
%Obersve that the vertex sums $\phi_f(v_1),\ldots, \phi_f(v_a)$ and $\phi_f(u_1),\ldots,\phi_f(u_b)$ are all distinct and between $-n_-$ and $n_+$.
By symmetry, we assume $n^+\ge n^-$ in the sequel.

\begin{figure}[ht]
\begin{center}
\begin{picture}(110,50)
\put(20,0){\circle*{5}}
\put(20,10){\circle*{5}}
\put(20,30){\circle*{5}}
\put(20,40){\circle*{5}}
\put(40,20){\line(-1,1){20}}
\put(40,20){\line(-2,1){20}}
\put(40,20){\line(-1,-1){20}}
\put(40,20){\line(-2,-1){20}}
\put(40,20){\circle*{5}}
\put(40,20){\line(1,0){40}}
\put(80,20){\circle*{5}}
\put(80,20){\line(1,1){20}}
\put(80,20){\line(2,1){20}}
\put(80,20){\line(1,-1){20}}
\put(80,20){\line(2,-1){20}}
\put(100,0){\circle*{5}}
\put(100,10){\circle*{5}}
\put(100,30){\circle*{5}}
\put(100,40){\circle*{5}}
\put(0,0){$v_a$}
\put(0,10){$\vdots$}
\put(0,30){$v_2$}
\put(0,40){$v_1$}
\put(110,0){$u_b$}
\put(110,10){$\vdots$}
\put(110,30){$u_2$}
\put(110,40){$u_1$}
\put(40,10){$v$}
\put(70,10){$u$}
\end{picture}
\end{center}
\caption{The double star $S_{a,b}$.}
\end{figure}
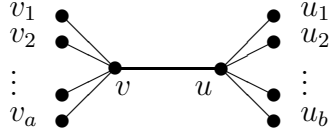

\begin{itemize}
\item {Case 1: $b\ge 2$.}

If $n_-=0$, then we assign $k+a+b+1,k+a+b,\ldots,k+1$  to edges
$vu$, $v_1v$, $u_1u$, $v_2v$, $u_2u,\ldots, v_bv$, $u_bu$, possibly $v_{b+1}v$, $v_{b+2}v,\ldots$ in order.
Since the largest label is assigned to $vu$ and $f(v_iv)> f(u_iu)$ for all $1\le i\le b$,
we have $\phi_f(v)>\phi_f(u)>\phi_f(w)$ for any pendent vertex $w$.
Now we suppose that both $n_+$ and $n_-$ are nonzero.

%$v_1v,v_2v,\ldots,v_av,u_1u,u_2u,\ldots,u_bu,vu$ in order;
%If both $n_+$ and $n_-$ are nonzero, then $0$ is also a usable label. %({\bf To be revised})
When $n_+ - n_-\ge 2$, first assign $\pm 1$, $\pm 2,\ldots,\pm n_-$ to a pair of edges $u_iu$'s or a pair of edges $v_iv$'s simultaneously. 
Let $H$ be the subgraph whose edge set is the set of all unlabeled edges. 
Observe that $H$ must be either a star or a double star. 
If $H$ is a star, then we assign the remaining labels to the edges of $H$ arbitrarily.
Then one of the vertex sums $\phi_f(u)$ and $\phi_f(v)$ is equal to $f(vu)$ and the other is equal to $(n_-+1)+\cdots+n_+> n_+$. 
The vertex sums of other pendent vertices are just the labels of the pendent edges. 
None of them is equal to $f(vu)$ and all are at most $n_+$.  
%This labeling gives vertex sums satisfying $|f(u)|>|f(w)|>|f(v)|=0$ or $|f(v)|>|f(w)|>|f(u)|=0$ for any $w\in\{u_1,\ldots, u_a,v_1,\ldots,v_b\}$.
If $H$ is a double star, then we label the edges of $H$ as the case of $n_-=0$.   
Without loss of generality, assume $d_H(v)\ge d_H(u)$.
First assign $n_+$ to $vu$, then assign $n_+-1$, $n_+-2,\ldots$ to an edge $v_iv$ and an edge $uu_j$ alternatively. 
We have $\phi_f(v)> \phi_f(u)\ge 2n_+-2>\phi_f(w)$ for any pendent vertex $w$.
%We label $vu$ by 0, the edges incident to $v$ by the labels of larger absolute values, 
%and the edges incident to $u$ by the remaining labels. 
%Both labelings satisfy $|\phi_f(v)|>|\phi_f(w)|>|\phi_f(u)|=0$ for any $w\in\{v_1,\ldots, v_a,u_1,\ldots,u_b\}$.
%Suppose that $H$ contains at least one vertex of degree two, we may assume that $d_H(u)=2$, which implies there is a unique unlabeled edge $uu_j$ for some $j$. 
%Then we label the edge $uu_j$ by 0, the edge $vu$ by the label of the largest absolute value, 
%and other edges of $H$ by the remaining labels arbitrarily. 
%The vertex sum $\phi_f$ satisfies $|\phi_f(v)|>|\phi_f(u)|>|\phi_f(w)|$ for any $w\in\{v_1,\ldots, v_a,u_1,\ldots,u_b\}$.

When $n_+ - n_-=1$, note that since $n_+ + n_-+1=|E(S_{a,b})|=a+b+1$ and $a\ge b$, we have $n_-\ge b\ge 2$. 
First, we assign 0 to $vu$ and $\pm 1$, $\pm 2,\ldots, \pm(n_--2)$ simultaneously to a pair of edges $u_iu$'s or a pair of edges $v_iv$'s, but 
leave three unlabeled edges incident to $v$ and two unlabeled edges incident to $u$, or vice versa. 
Now we assign $n_+$, $n_-$, and $n_--1$ to the three unlabeled edges incident to the same vertex, 
and $-n_-$, and $-(n_--1)$ to the remaining edges.
We have  $\{\phi_f(u), \phi_f(v)\}=\{3n_+-3,-(2n_--1)\}$, so the vertex sums of $u$ and $v$ are different from that of any pendent vertex.

When $n_+=n_-$, we assign 0 to $vu$ and $-n_-$, $-(n_--1),\ldots, -1$, $1$, $2,\ldots, n_+$ to the edges $u_1u$, $u_2u,\ldots$, $u_bu$, $v_1v$, $v_2v,\ldots,v_av$ accordingly.
Note that $b\ge 2$ implies $\phi_f(u)\le -n_--(n_--1)$ and $\phi_f(v)\ge n_++(n_+-1)$.
Then labeling $f$ gives $\phi_f(u)< \phi_f(w) <\phi_f(v)$ for any pendent vertex $w$. 
  
%For the first two cases, we assign 0 to the edge $vu$ and the remaining labels to the unlabeled edges.
%Then for the two internal vertices $u$ and $v$, one has the vertex sum 0 and the other has the vertex sum whose absolute value is greater than the absolute values of all nonzero labels, which are the vertex sums of the vertices $v_1,\ldots,v_a,u_1,\ldots,u_b$.
%For the third case, suppose that the unlabeled edges, including $vu$, are $v_av,vu,u_bu$.
%We label the three edges by the labels with the largest absolute value, the second largest absolute value,
%and 0, respectively.
%Then the vertex sums of the vertices satisfy $|\phi_f(v)|>|\phi_f(v_a)|>|\phi_f(u)|>|\phi_f(w)|>|\phi_f(v)|=0$
%for any $w\in \{v_1,\ldots,v_{a-1},u_1,\ldots,u_{b-1}\}$.
%When $n_+=n_-=a=b$, we assign 0 to $vu$, all positive labels to edges $vv_i$ and all negative labels to edges $uu_i$.
%It is easy to see that $\phi_f(v)> \phi_f(w)>\phi_f(u)$ for any $w\in \{v_1,\ldots,v_a,u_1,\ldots,u_b\}$.
%Therefore, $S_{a,b}$ is absolutely antimagic if $b\ge 2$.

\item {Case 2: $b=1$.} 

%First, consider $S_{2i-1,1}$ and $i\ge 2$. If $k\neq -(i+1)$, then $n_+\neq n_-$.
If $n_-=0$ , we then assign $k+a+2$, $k+a+1,\ldots,k+1$ to the edges
$vu$, $v_1v$, $u_1u$, $v_2v$, $v_3v,\ldots, v_av$ in order. By the argument in Case 1,
this is a $k$-shifted-antimagic labeling for $S_{a,1}$.

When $n_+-n_-\ge 2$, we assign $n_+$ to $vu$, 0 to $u_1u$, and 
other labels to the remaining edges.
This labeling gives $\phi_f(v)>\phi_f(u)=n_+>\phi_f(w)$ for any pendent vertex $w$.

When $n_+-n_-=1$, since $n_++n_-+1=a+2$, this implies $a$ is even and $n_-\ge 2$. 
Let us first assume $a\neq 2$. Then assign $-(n_--1)$ to $vu$,  
$-n_-$ to $u_1u$, and other labels to the remaining edges.
We have $\phi_f(v)=n_++n_->\phi_f(w)>-(2n_--1)=\phi_f(u)$ for any pendent vertex $w$.

When $n_+=n_-$, $a$ is odd and $k=-\frac{1}{2}(a+3)$
We show that $S_{a,1}$ is not $k$-shifted-antimagic. 
It is easy to see that $0$ cannot be assigned to the internal edge $vu$ or the edge $u_1u$,
because the former leads to $\phi_f(u)=\phi_f(u_1)=f(u_1u)$ and the latter leads to $\phi_f(v)=\phi_f(u_1)=0$.
Thus, $0$ only goes to some pendent edge $v_iv$.
Now that $f(u_1u)=x\neq 0$, if $f(vu)=-x$, then $\phi_f(u)=\phi_f(v_1)=0$ which is forbidden.
Else, we have $f(v_jv)=-x$ for some $j\neq i$, and then the labels of edges incident to $v$, except for $vv_j$, sum to zero.
Therefore $\phi_f(v)=\phi_f(v_j)=-x$, which is also a contradiction.
Consequently, $S_{a,1}$ is not $-k$-shifted-antimagic for $k=-\frac{1}{2}(a+3)$.
\end{itemize}

Here comes the discussion of $k$-shifted-antimagic labelings on $S_{1,1}$ (or equivalently $P_4$) and $S_{2,1}$.
For $S_{1,1}$, setting $f(v_1v)=k$, $f(vu)=k+2$,
and $f(u_1u)=k+1$ for $k\ge -1$ yields a $k$-shifted-antimagic labeling.
By Proposition~\ref{PROP:SYM}, $S_{1,1}$ is $k$-shifted-antimagic for $k\le -3$.
When $k=-2$, observe that 0 cannot be assigned to the internal edge $vu$ and the $\pm 1$ cannot be assigned to two incident edges, thereby  $(-1)$-shifted-antimagic labelings do not exist.
For $S_{2,1}$, setting $f(v_1v)=k+2, f(v_2v)=k+3$, $f(vu)=k+4$,
and $f(u_1u)=k+1$ for $k\ge -1$ obtains a $k$-shifted-antimagic labeling, and again $S_{2,1}$ is $k$-shifted-antimagic for $k\le -4$
by Proposition~\ref{PROP:SYM}.
When $k=-2$ (the argument of $k=-3$ is similar by symmetry), it is easily seen that 0 cannot be assigned to the internal edge $vu$.
If $0$ is assigned to $uu_1$,
then the two vertices incident to the edges labeled by 2 have the same vertex sum.
Thus, 0 can only be assigned to one of $v_1v$ and $v_2v$.
In this case, there is either another vertex whose vertex sum equals 0 or a pair of vertices whose vertex sum are equal to 1, no matter how we label the remaining three edges.
According to the discussion above, no $k$-shifted-antimagic labeling exists for $k\in\{-2,-3\}$.
\end{proof}

For trees of diameter four, we have the following partial results.

\begin{theorem}
The path $P_5$ is $k$-shifted-antimagic if and only if $k\not\in\{ -2,-3\}$.
\end{theorem}

\begin{proof}
Denote $P_5$ by $v_1v_2v_3v_4v_5$, and let $e_i=v_iv_{i+1}$.
By the claim in Theorem~\ref{THM:PAA} and Proposition~\ref{PROP:SYM},
$P_5$ is $k$-shifted-antimagic for $k\not\in\{-4,-3,-2,-1\}$.

To show that $P_5$ is $(-1)$-shifted-antimagic,
we label $v_1v_2,v_2v_3,v_3v_4,v_4v_5$, by $0,1,3,2$, respectively.
It follows by symmetry that $P_5$ is $(-4)$-shifted-antimagic.

It suffices to show that $P_5$ is not $k$-shifted-antimagic for $k=-2$ ($k=-3$ follows by symmetry).  Assume that $f$ is a $(-2)$-shifted-antimagic labeling of $P_5$.
Then $\{f(v_iv_{i+1})\mid 1\le i \le 4\}=\{-1,0,1,2\}$.
If $f(v_2v_3)=0$, then $\phi_{f}(v_1)=\phi_{f}(v_2)$, which is not allowed.
By symmetry, $f(v_3v_4)\neq 0$ and we may assume $f(v_1v_2)=0$.
If $-1$ and $1$ are labeled to two incident edges,
then the common vertex of the two edges has a vertex sum $0$,
which is the same as $\phi_f(v_1)$, a contradiction.
The remaining possible labelings are $f(v_1v_2)=0,f(v_2v_3)=1,f(v_3v_4)=2,f(v_4v_5)=-1$, or
$f(v_1v_2)=0,f(v_2v_3)=-1,f(v_3v_4)=2,f(v_4v_5)=1$.
In either case, we have two vertices whose vertex sums are equal to one.
Thus, the labeling $f$ does not exist.
\end{proof}

Define $P_5'$ to be the graph obtained by attaching an edge to the central vertex of $P_5$, as
illustrated in Figure~\ref{FIG:PFP}.

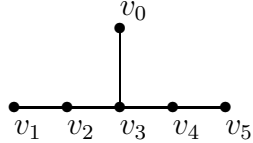
\begin{figure}[ht]
\begin{center}
\begin{picture}(100,50)
\put(10,10){\circle*{4}}
\put(30,10){\circle*{4}}
\put(50,10){\circle*{4}}
\put(70,10){\circle*{4}}
\put(90,10){\circle*{4}}

%\put(15,15){$e_1$}
%\put(35,15){$e_2$}
%\put(55,15){$e_3$}
%\put(75,15){$e_4$}
%\put(52,25){$e'$}

\put(50,40){\circle*{4}}

\put(10,0){$v_1$}
\put(30,0){$v_2$}
\put(50,0){$v_3$}
\put(70,0){$v_4$}
\put(90,0){$v_5$}
\put(50,45){$v_0$}

\put(10,10){\line(1,0){80}}
\put(50,10){\line(0,1){30}}
\end{picture}
\end{center}
\caption{The graph $P_5'$.}\label{FIG:PFP}
\end{figure}

\begin{theorem}
The graph $P_5'$ is $k$-shifted-antimagic if and only if $k\neq -3$.
\end{theorem}

\begin{proof}
We first label $v_1v_2,v_2v_3,v_3v_4,v_4v_5$, and $v_0v_3$ by $2,4,5,3$, and $1$, respectively.
This is a strongly antimagic labeling on $P_5'$.
Thus, what remain to discuss are the cases for $k\in\{-1,-2,-3,-4,-5\}$.
A $(-1)$-shifted-antimagic labeling for $P_5'$ can be constructed by extending an antimagic labeling on $P_5$ to $P_5'$ by giving 0 to the edge $v_0v_3$.
Next, we can label $v_1v_2,v_2v_3,v_3v_4,v_4v_5$, and $v_0v_3$ by $3,2,1,0$, and $-1$, respectively,
to get a $(-2)$-shifted-antimagic labeling.
The existence of $k$-shifted-antimagic labeling for $k\in\{-4,-5\}$ follows by Proposition~\ref{PROP:SYM} as before.

We now show that $P_5'$ is not $(-3)$-shifted-antimagic.
Clearly, when labeling the edges by $-2,-1,0,1,2$, the number $0$ can only be labeled to $v_1v_2$, $v_4v_5$, or $v_0v_3$.
When assigning $0$ to $v_1v_2$ (the case of $v_4v_5$ is similar), then the labels of $v_2v_3$, $v_3v_4$, and $v_0v_3$ must contain a pair of labels $x$ and $-x$.
If this pair happens at $v_2v_3$ and $v_3v_4$, then the vertex sums of $v_0$ and $v_3$ are equal;
else if the pair happens at $v_3v_4$ and $v_0v_3$, then the vertex sums of $v_2$ and $v_3$ are equal;
else if the pair happens at $v_2v_3$ and $v_0v_3$, then the vertex sums of $v_4$ is zero,
since the labels of $v_3v_4$ and $v_4v_5$ must be different by a negative sign, which is equal to that of $v_1$.
Therefore, $0$ only goes to $v_0v_3$.
Since the vertex sum of $v_0$ is zero, it is forbidden to assign two labels $x$ and $-x$ to two incident edges in $v_1v_2,v_2v_3,v_3v_4,v_4v_5$. We may assume the labels of $v_1v_2,v_2v_3,v_3v_4$, and $v_4v_5$ are $x,y,-x$, and $-y$, respectively, and $x\neq \pm y$.
Then the vertex sums of $v_1,\ldots,v_5$ are $x,x+y,y-x,-x-y,-y$, respectively.
Notice that $\{\pm x,\pm y\}=\{\pm 1,\pm 2\}$, so one of $x=2y$, $x=-2y$, $y=2x$, and $y=-2x$ must hold.
However, any one of the equations leads to a pair of equal numbers in $\{x,x+y,y-x,-x-y,-y\}$.
Consequently, no $(-3)$-shifted-antimagic labeling exists for $P_5'$.
\end{proof}

\subsection{Disconnected graphs}

The study of antimagic labelings on disconnected graphs does not draw as much attention as connected graphs.
In fact, there are abundant examples of disconnected, $k$-shifted-antimagic graphs with restricted values $k$.
%For connected graphs, no integer $k$ exists for which $P_2$ is $k$-shifted-antimagic.
In Subsection 3.1, we have seen several examples of connected graphs that are $k$-shifted-antimagic for all but merely one or two excluded integers $k$, in contrast to disconnected graphs, we show that there are $k$-shifted-antimagic graphs for
all but arbitrarily many excluded integers.

Next, we introduce some non-$0$-shifted-antimagic (i.e. non-antimagic) disconnected graphs.
Let $G_1+G_2+\cdots +G_c$ be a graph consisting of $c$ connected components $G_1,\ldots,G_c$.
Denote it as $cG$ if all components are isomorphic to $G$. Shang, Lin, and Liaw~\cite{SLL2015} prove that if $G=cP_3+S_n$ is antimagic if and only if
$c\le \min\{2n+1,\frac{2n-5+\sqrt{8n^2-24n+17}}{2}\}$. In \cite{S2015},  Shang pointed out that $cP_3$ is antimagic if and only if $c=1$. In the following, we show that if a graph contains too many components isomorphic to $P_3$, then it is not antimagic.

\begin{theorem}\label{THM:GP3}
For any graph $G$, there exists a constant $c=c(G)$ such that the graph $G+cP_3$ is not antimagic.
\end{theorem}

\begin{proof}
Given a graph $G=(V,E)$ with $|V(G)|=n$ and $|E(G)|=m$,
the graph $G+cP_3$ contains $n+3c$ vertices and $m+2c$ edges.
Denote the $c$ paths by $u_1v_1w_1, u_2v_2w_2,\ldots,u_cv_cw_c $.
Assume that $f$ is an antimagic labeling of $G+cP_3$.
Then we have $\phi_f(u_i)=f(u_iv_i)$ and $\phi_f(w_i)=f(v_iw_i)$ for $1\le i\le c$,
which are all distinct.
Thus, either $\phi_f(v_i)\in \{1,2\ldots, m+2c\}\setminus\{f(u_iv_i),f(v_iw_i)\mid 1\le i\le c\}$,
or $\phi_f(v_i)>m+2c$. When $c>m$, $\phi_f(v_i)>m+2c$ holds for at least $c-m$ $v_i$'s.
Note that $\phi_f(v_i)=\phi_f(u_i)+\phi_f(v_w)$.
Thus, the total vertex sum of these $v_i$'s with $\phi_f(v_i)>m+2c$ and their neighbors $u_i$'s and $w_i$'s is at least
$2[(m+2c+1)+(m+2c+2)+\cdots+ (m+2c+(c-m))]=(1+m+5c)(c-m)$.
Observe that the total vertex sum of all vertices is $\sum_{v\in V(G)}\phi_f(v)=2\sum_{e\in E(G)} f(e)=(1+m+2c)(m+2c)$.
We then compare it with the weight sum of the vertices on the $P_3$'s.
Observe that
\begin{eqnarray*}
(1+m+2c)(m+2c)
&=& 1+m+m^2+2c+4mc+4c^2\\
&<& -m-m^2+(1-4m)c+5c^2\\
&=&(1+m+5c)(c-m)
\end{eqnarray*}
for sufficiently large $c$. This leads to a contradiction.
Consequently, $G+cP_3$ is not antimagic for sufficiently large $c$.
\end{proof}

The following result points out that $cP_3$ is an example of the $k$-shifted-antimagic graphs for all but finitely many $k$.
In addition, the number of values $k$ for which $cP_3$ is not $k$-shifted-antimagic increases as $c$ increases, and they form a set of consecutive integers.

\begin{theorem} The graph $cP_3$ is $k$-shifted-antimagic if and only if
$k\not\in\{-\lfloor\frac{5c}{2}\rfloor,-\lfloor\frac{5c}{2}\rfloor+1,\ldots, \lfloor \frac{c}{2}\rfloor-1\}$.
\end{theorem}

\begin{proof}
We show that $cP_3$ has a $k$-shifted-antimagic labeling if and only if
$k\ge \lfloor \frac{c}{2}\rfloor$ or $k\le -(\lfloor\frac{5c}{2}\rfloor +1)$.
As in the previous theorem, let the $c$ copies of $P_3$'s be $u_1v_1w_1,\ldots, u_cv_cw_c$.

We first prove that there exists a $k$-shifted-antimagic labeling of $cP_3$ for $k=\lfloor\frac{c}{2}\rfloor$.
If $c=2\ell +1$, then all the labels are elements in $\{\ell+1,\ldots,5\ell+2\}$.
Assign each of the following pairs of labels to the edges in one $P_3$:
$\{\ell+1,4\ell+2\},\{\ell+3,4\ell+1\},\ldots,\{\ell+2i+1,4\ell+3-i\},\ldots,\{3\ell+1,3\ell+2\}$, and $\{\ell+2,5\ell+2\},\{\ell+4,5\ell+1\},\ldots,\{\ell+2i,5\ell+3-i\},\ldots,\{3\ell,4\ell+3\}$.
Note that the vertex sums of $v_i$'s are $5\ell+3,5\ell+4,\ldots, 7\ell+3$, which are all distinct and greater than $5\ell+2$. So this is a $\lfloor\frac{c}{2}\rfloor$-shifted-antimagic labeling of $cP_3$ for odd $c$.

If $c=2\ell$, then we give each of the following pairs of labels to the edges in one $P_3$:
$\{\ell+1,4\ell\},\{\ell+3,4\ell-1\},\ldots,\{\ell+2i+1,4\ell+1-i\},\ldots,\{3\ell-1,3\ell+1\}$, and $\{\ell+2,5\ell\},\{\ell+4,5\ell-1\},\ldots,\{\ell+2i,5\ell+1-i\},\ldots,\{3\ell,4\ell+1\}$.
Note that the vertex sums of $v_i$'s are $5\ell+1,5\ell+2,\ldots, 6\ell$ and $6\ell+2,6\ell+3,\ldots,7\ell+1$,
which are all distinct and greater than $5\ell+2$. Again, this is a $\lfloor\frac{c}{2}\rfloor$-shifted-antimagic labeling of $cP_3$ for even $c$.

Observe that the above labeling is indeed a strongly antimagic labeling. Hence there exists a $k$-antimagic labeling for every $k\ge \lfloor\frac{c}{2}\rfloor$.

Suppose that the injection $f$ from $\{u_iv_i,v_iw_i\mid 1\le i\le c\}$ to $\{k+1,\ldots, k+2c\}$
is a $k$-shifted-antimagic labeling on $cP_3$. Since each edge is incident to a vertex of degree one,
the vertex sums of the vertices of degree one are exactly the labels of the edges.
Hence, we have $\phi_f(v_i)>k+2c$ for all $i$.

Assume that $cP_3$ is $k$-shifted-antimagic for some $0\le k\le \lfloor\frac{c}{2}\rfloor-1$,
and let $f$ be such a labeling.
Without loss of generality, assume $\phi_f(v_1)<\phi_f(v_2)<\cdots <\phi_f(v_c)$,
and $\phi_f(v_1)>k+2c$.
We have
\[\sum_{i=1}^{c}(k+2c+i)\le \sum_{i=1}^c \phi_f(v_i)=\sum_{e\in E(cP_3)}f(e)=\sum_{i=1}^{2c}(k+i).
\]
Simplifying the above inequality leads to $k\ge \frac{c-1}{2}$, which is essentially the same as $k\ge \lfloor \frac{c}{2}\rfloor$.
By Proposition~\ref{PROP:SYM}, when $k$ is negative, $cP_3$ is $k$-shifted-antimagic if and only if
$-(m+k+1)\ge \lfloor\frac{c}{2}\rfloor$.
Equivalently, $k\le -m-1-\lfloor\frac{c}{2}\rfloor=-\lfloor\frac{5c}{2}\rfloor-1$.
\end{proof}

%\begin{theorem} The graph $P_3+P_4$ is $k$-shifted-antimagic for $k\not\in\{-1,-2,-3-,4,-5\}$.
%\end{theorem}

A $\{P_{k_1},\ldots, P_{k_s}\}$-free linear forest is a disjoint union of paths such that none of the paths is $P_{k_i}$ for any $i$.
It is conjectured by Shang~\cite{S20152} that every $\{P_2,P_3\}$-free linear forest is antimagic.
The following result demonstrates an example of a $\{P_2,P_3\}$-free linear forest that is not $k$-shifted-antimagic with nonconsecutive $k$'s.

\begin{theorem} The graph $2P_4$ is $k$-shifted-antimagic if and only if $k\not\in\{-2,-5\}$.
\end{theorem}

\begin{proof}
By Proposition~\ref{PROP:SYM}, it suffices to verify that $G=2P_4$ is $k$-shifted-antimagic
for $k=-3$ and $k\ge -1$, but not $(-2)$-shifted-antimagic. Denote the two paths by $v_1v_2v_3v_4$ and $u_1u_2u_3u_4$.

We first show that it is not $(-2)$-shifted-antimagic.
Suppose that $f$ is a $(-2)$-shifted-antimagic labeling with labels $-1$, $0$, $1$, $2$, $3$, and $4$.
If $f(v_2v_3)=0$, then not only the vertex sums $\phi_f(v_1)=\phi_f(v_2)$  but also $\phi_f(v_3)=\phi_f(v_4)$.
Hence $f(v_2v_3)\neq 0$. Similarly, $f(u_2u_3)\neq 0$
Without loss of generality, we may assume $f(v_1v_2)=0$.
%Let us consider which edge is labelled by $-1$.

\begin{itemize}
\item {Case 1: $f(v_2v_3)=-1$.}\\
 Let $f(v_3v_4)=x$. If $x=1$, then $\phi_f(v_3)=\phi_f(v_1)=0$.
If $x=2$, then $\phi_f(v_3)=1$, and $1$ must be labeled to the internal edge of the other $P_4$.
Thus, $\{f(u_1u_2),f(u_3u_4)\}=\{3,4\}$. Consequently, there exist an internal vertex and a leaf of the same vertex sum in this path.
If $x=3$, then $\{\phi_f(v_3),\phi_f(v_4)\}=\{2,3\}$.
To avoid the same vertex sum, neither can we label two incident edges in the other $P_4$ by 1 and 2,
nor can we label a pendent edge by 2. So, $x\neq 3$.
The last possibility is $x=4$. Then we have $\{\phi_f(v_3),\phi_f(v_4)\}=\{3,4\}$.
The edge labeled by 1 in the other $P_4$ cannot be incident to the edge labeled by 2 or 3.
As a consequence, $x\neq 4$.

\item {Case 2: $f(v_3v_4)=-1$.}\\
 The argument is exactly the same as the argument above.

\item {Case 3: $f(u_1u_2)=-1$ (or $f(u_3u_4)=-1$).}\\
 Let $f(u_2u_3)=x$. If $x=1$, then $\phi_f(u_2)=\phi_f(v_1)=0$.
When $x\ge 2$, one of the vertices $v_2,v_4,u_4$ must have vertex sum $x-1$.
So we exclude this case.

\item {Case 4: $f(u_2u_3)=-1$.}\\
 No matter how the remaining edges are labeled,
$\{\phi_f(v_2),\phi_f(v_4),\phi_f(u_1),\phi_f(u_4)\}=\{1,2,3,4\}$
and $\{\phi_f(u_2),\phi_f(u_3)\}\subset \{0,1,2,3\}$.
Thus, $\{\phi_f(v_2),\phi_f(v_4),\phi_f(u_1),\phi_f(u_4)\}\cap\{\phi_f(u_2),\phi_f(u_3)\}\neq\emptyset$.
\end{itemize}
So, $f$ cannot be a $(-2)$-shifted-antimagic labeling.

For $k=-1$, we label the edges as follows:
$f(v_1v_2)=0$, $f(v_2v_3)=4$, $f(v_3v_4)=1$, $f(u_1u_2)=2$, $f(u_2u_3)=5$, $f(u_3u_4)=3$.
Observe that $f$ is a strongly antimagic labeling. Hence $2P_4$ is $k$-shifted-antimagic for all $k\ge -1$.

Finally, for $k=-3$, we label $v_1v_2$, $v_2v_3$, and $v_3v_4$, by $-2$, $-1$, and $0$, respectively,
and also label $u_1u_2$, $u_2u_3$, and $u_3u_4$, by $2$, $3$, and $1$, respectively.
It is easy to verify that all vertex sums are all distinct. The proof is complete.
\end{proof}

In addition to $2P_4$, we have another graph which is not $k$-shifted-antimagic with nonconsecutive $k$'s.

\begin{theorem} The graph $2S_3$ is $k$-shifted-antimagic if and only if $k\not\in\{-2,-5\}$.
\end{theorem}

\begin{proof} Denote the vertices of one star by $v,v_1,v_2,v_3$ and those of the other star by $u,u_1,u_2,u_3$, where $v_i$'s and $u_i$'s are the leaves.
Let us label the edges $v_1v,v_2,v,v_3v$ by $k+1,k+3,k+6$, and the edges $u_1u,u_2,u,u_3u$ by $k+2,k+4,k+5$, respectively.
Then $\phi_f(u)>\phi_f(v)=3k+10>k+6$ if and only if $k>-2$. Hence $2S_3$ is $k$-shifted-antimagic for all $k>-2$,
and by Proposition~\ref{PROP:SYM}, for all $k<-5$. For $k=-3$, we assign $-2,-1,0$ to the edges of one star 
and $1,2,3$ to the edges of the other star. This is clearly a $(-3)$-shifted-antimagic labeling on $2S_3$. 
Again, by Proposition~\ref{PROP:SYM}, $2S_3$ is $(-4)$-shifted-antimagic. 
Finally, we show that $2S_3$ is not $k$-shifted-antimagic for $k\in\{-2,-5\}$. 
We only show the case $k=-2$.
When $k=-2$, the usable labels are $-1,0,1,2,3,4$. For any labeling $f$, we have $0\le \phi_f(v)\le 9$, $0\le \phi_f(u)\le 9$, and $\phi_f(v)+\phi_f(u)=9$. 
Therefore, at least one of $\phi_f(v)$ and $\phi_f(u)$ is less than 5, which is equal to the vertex sum of some leaf. 
So $2S_3$ is not $(-2)$-shifted-antimagic.
 \end{proof}

%\remark The integers $k$ for which $P_4+P_4$ is not $k$-shifted-antimagic do not appear consecutively.

\section{Concluding remarks and future work}

%In this paper, we proposed the question that for which $k$, we can label a graph $G$ on $m$ edges
%exactly by $k+1,k+2,\ldots, k+m$ such that the vertex sums are all distinct.
%We have proved that a graph $G$ is $k$-shifted-antimagic for all sufficiently large $|k|$ when
%$G$ is either a forest or a graph without vertex of even degree and no any component in $G$ is isomorphic to $K_2$.
%Recall that Hefetz~\cite{H2005} and Wong and Zhu~\cite{WZ2012} proposed another $k$-antimagic labeling by allowing the labels to be $1,2,\ldots,k+m$.
%Since our labels are more restrictive,
%the existence of our $k$-shifted antimagic labeling implies the existence of their $k$-antimagic labeling of the graphs.

We first want to stress that the ``level-by-level'' algorithm mentioned in Section 2 could be applicable
to graphs containing vertices of even degree with some extra conditions.
For example, if all vertices of even degree have distinct degrees,
then we can prove the existence of an $SDDS$-labeling of such a graph using the same argument in the proof of Theorem~\ref{THM:ODD}.
Alternatively, we may assume that all vertices of even degree have the same degree,
and then apply a minimum $d$-covering pair method by Chang et al.~\cite{CLPZ2016} used to find
the antimagic labeling of even regular graphs.
Readers may consult~\cite{CLPZ2016} for more details.
Unfortunately, we are not able to generalize their method to non-regular graphs.
Despite the fact that the existence of the $k$-shifted antimagic labeling for graphs with mixed even degrees
is still unknown,
we believe that all graphs are $k$-shifted-antimagic for some $k$ and pose the following conjecture:

\begin{conjecture}\label{CONJ:4}
Every graph is $k$-shifted-antimagic for $|k|$ sufficiently large if it
does not contain a component isomorphic to $P_2$.
\end{conjecture}

Theorem~\ref{THM:GP3} states that for any graph $G$,
the graph $G+cP_3$ is not antimagic for sufficiently large $c$. In other words, the union of antimagic graphs could be not antimagic.
We are curious that if there is any other type of graphs $G=G_1+G_2+\cdots +G_k$ such that
each $G_i$ is antimagic but $G$ is not?
%It is clear that if each $G_i$ is absoultely antimagic, then $G$ is also absoultely antimagic.
%Conversely, if a graph $G=G_1+G_2+\cdots +G_k$ is not antimagic or not $k$-shifted antimagic for other $k\neq 0$,
%then there must exist some $G_i$ which is not absolutely antimagic.

If a counterexample exists for Conjecture~\ref{CONJ:4}, then the graph must contain a component which is not absolutely antimagic.
The reason is that if each $G_i$ is absolutely antimagic,
then $G$ is $k$-shifted-antimagic for sufficiently large $|k|$.
To see this, we first find an antimagic labeling of $G_1$ and a
$k$-shifted-antimagic labeling of $G_i$ with $k=\sum_{j=1}^{i-1}|E(G_j)|$ for $i\ge 2$.
These labelings together form an SDDS-labeling of $G$,
and, by Lemma~\ref{LEM:SFT}, $G$ is $k$-shifted for sufficiently large $k$.
Thus, a more fundamental question is whether or not a connected graph is absolutely antimagic.

In Subsection 3.1, some trees which are not absolutely antimagic were presented.
We inspected some small examples of trees of diameter at least five and graphs containing a cycle,
but it turns out that all those small graphs are absolutely antimagic.
It may be interesting to answer the following two questions.

\begin{question}
Find a tree $T$ of diameter at least five which is not $k$-shifted-antimagic for some integer $k$.
\end{question}

\begin{question}
Find a graph $G$ containing a cycle, which is not $k$-shifted-antimagic for some integer $k$.
\end{question}

% and
%$P_4+P_4$ is not $k$-shifted-antimagic for $k\in\{-2,-5\}$.
%We are curious that if there are any other antimagic graphs $G$ and $H$ and the integers $k$
%such that $G+H$ is not $k$-shifted antimagic?

\end{document}